\documentclass[a4paper,11pt]{article}

\usepackage{amsmath,amssymb,amsthm,fullpage}
\usepackage{enumerate,tikz}
\usetikzlibrary{decorations.pathreplacing}

\title{A time-invariant random graph with splitting events}

\author{
  Agelos~Georgakopoulos\footnote{University of Warwick. \texttt{j.haslegrave@cantab.net}}
  \and
  John~Haslegrave\footnotemark[1]}

\newtheorem{theorem}{Theorem}[section]
\newtheorem{lemma}[theorem]{Lemma}
\newtheorem{corollary}[theorem]{Corollary}
\newtheorem{proposition}[theorem]{Proposition}
\newtheorem{definition}[theorem]{Definition}
\newtheorem{remark}[theorem]{Remark}
\newtheorem{claim}[theorem]{Claim}

\newenvironment{poc}{\begin{proof}[Proof of Claim]}{\end{proof}}

\newcommand{\eps}{\varepsilon}
\newcommand*{\cmot}[1][t]{G^\circ_{#1}}
\newcommand*{\dmot}[1][t]{G^{\circ\circ}_{#1}}
\newcommand*{\dmotf}[1][t]{\overline{G}^{\circ\circ}_{#1}}
\newcommand*{\mean}[1]{\mathbb{E}(#1)}
\newcommand*{\prob}[1]{\mathbb{P}(#1)}
\newcommand*{\abs}[1]{\lvert #1\rvert}
\newcommand*{\floor}[1]{\lfloor #1\rfloor}
\newcommand*{\ceil}[1]{\lceil#1\rceil}
\newcommand*{\ee}[1]{\mathrm{e}^{#1}}
\newcommand*{\Po}[1]{\operatorname{Po}(#1)}
\newcommand*{\cmp}[1]{#1^{\scriptscriptstyle\complement}}

\newcommand{\clock}{\operatorname{Exp}(1)}
\newcommand*{\gft}[1][t]{\overline{G}_{#1}}
\newcommand*{\gftc}[1][t]{\overline{G}^\circ_{#1}}

\newcommand{\Lr}[1]{Lemma~\ref{#1}}
\newcommand{\Tr}[1]{Theorem~\ref{#1}}
\newcommand{\Sr}[1]{Section~\ref{#1}}
\newcommand{\Dr}[1]{Definition~\ref{#1}}

\newcommand{\defi}[1]{{\emph{#1}}}

\newcommand{\ie}{i.e.\ }

\newcommand{\iid}{i.i.d.\ }

\newcommand{\la}{\lambda}
\newcommand{\lest}{\leq_{\mathrm{st}}}

\begin{document}
\maketitle 
\begin{abstract}We introduce a process where a connected rooted multigraph evolves by splitting events on its vertices, occurring randomly in continuous time. When a vertex splits, its incoming edges are randomly assigned between its offspring and a Poisson random number of edges are added between them. The process is parametrised by a positive real $\la$ which governs the limiting average degree. We show that for each value of $\la$ there is a unique random connected rooted multigraph $M(\la)$ invariant under this evolution. As a consequence, starting from any finite graph $G$ the process will almost surely converge in distribution to $M(\la)$, which does not depend on $G$. We show that this limit has finite expected size. 
The same process naturally extends to one in which connectedness is not necessarily preserved, and we give a sharp threshold for connectedness of this version. 

This is an asynchronous version, which is more realistic from the real-world network point of view, of a process we studied in \cite{gwrg,PIGG}.


\end{abstract}

\section{Introduction}
We consider a random network model with reproduction which evolves in continuous time. Each vertex independently, at rate $1$, splits into two. When a vertex splits, each of its existing edges is randomly rerouted to one of the two vertices produced, and these two vertices are connected by a random number of edges with distribution $\Po{\la/2}$, where $\la>0$ is a fixed parameter. If the resulting graph is disconnected, only the component of the root is retained (the precise definition is given in the next section). We  show that there is a unique random multigraph $M(\la)$ which is time-invariant under this evolution and has finite average degree (\Tr{limit}), and analyse some of its properties. As a consequence, if we run our process starting from any finite graph $G$, it will  almost surely converge in distribution to  $M(\la)$.

This model arose naturally in our recent work \cite{PIGG}: there, we considered the variant of the above evolution where all vertices split simultaneously in regular time intervals. We observed that there is a unique finite-degree random multigraph $G(\la)$ which is time-invariant under this evolution too. We will refer to $G(\la)$ as the \defi{synchronous version} of $M(\la)$. Moreover, we showed that  $G(\la)$ is identically distributed with the cluster of the origin in an instance of long-range percolation on the infinitely-generated group $\bigoplus_{i\in\mathbb{N}}\mathbb{Z}_2$. Perhaps surprisingly, given its alternative definition as a cluster of a percolation model on a group, and given that most percolation models on finitely generated groups undergo a phase transition \cite{DCGRSY},  $G(\la)$ is almost surely finite for any value of the intensity $\la$, and its expected size is finite. In this paper we show the analogous result for $M(\la)$ (\Tr{finite}).

Our splits can be thought of as reproduction of vertices, in the sense that a vertex produces a child and then passes on some of its connections to its child. In this sense, our first definition of $G(\la)$ is reminiscent of the models for random reproducing graphs studied by Jordan \cite{Jor11}, building on earlier deterministic models for social networks \cite{SC10,BHHPW}, with the key distinction being that in Jordan's model all connections of the parent are retained, whether or not they are inherited by the child. 

However, simultaneous, discrete-time reproduction by the whole population is not a realistic model for real-life networks. It is therefore natural to consider a variant in which reproduction events are independent and may occur at any time, which is part of the motivation of the current paper. Mechanisms for growing networks based on repeated vertex duplications have previously been proposed as plausible for the development of the web graph \cite{KRRSTU} and for evolution of biochemical networks \cite{BGD02,VFMV}. Mathematical analysis of such a model was carried out non-rigorously by Pastor-Satorras, Smith and Sole \cite{PSSS}, suggesting a limiting degree distribution which is power-law with an exponential cutoff, although subsequent rigorous work by Bebek, Berenbrink, Cooper, Friedetzky, Nadeau, and Sahinalp \cite{BBCFNS} showed that this is not the case. Another related model, motivated by duplication of genetic material, has been studied by Th\"ornblad \cite{Thorn} and by Backhausz and M\'ori \cite{BM16}; however, the graph structure of this model is particularly simple, being a collection of disjoint cliques. A similar model for a fixed population size, which has richer behaviour owing to the random loss of individual edges, was introduced by Bienvenu, D\'ebarre, and Lambert \cite{BDL19}.

Although the continuous-time model $M(\la)$ studied here is more natural in certain respects, its analysis is significantly more challenging than that of the synchronous version $G(\la)$ for the following reason. A basic tool in the analysis of both models is the underlying \defi{genealogical tree} $T$, containing all vertices in our evolution, and joining each vertex to its children by an edge. Starting with $T$, we can alternatively define our random graphs by joining pairs of leaves of $T$ with random independent edges with appropriately chosen probabilities. In the synchronous case, this $T$ is very simple: it is a binary tree of depth $n$ when we run the process for $n$ steps starting from a single vertex, and it is the so-called canopy tree when we start with $G(\la)$. When we start with $M(\la)$ however, $T$ is a random tree with a non-trivial distribution: it can be thought of as the local limit of the ball $B(t)$ of radius $t$ in first passage percolation on the full binary tree after re-rooting $B(t)$ at a leaf (see \Sr{secAsync} for more details). Thus our main results \Tr{limit} and \Tr{finite} below were much harder to prove than their analogues in \cite{PIGG}.

\subsection{Model and results}
It will be convenient for some proofs and statements of results to define both the main process defined above and a ``full'' version of the process in which other components are not discarded. In fact it is simpler to define the latter first. A \defi{multigraph} is a graph in which two vertices may be joined by several parallel edges. The multigraphs of this paper do not have loops, i.e.\ edges that start and end at the same vertex. 

\begin{definition}\label{full}For a rooted connected multigraph, $(G,o)$, the \emph{full process} $(\gft,o_t)_{t\geq 0}$ with parameter $\la>0$ is defined as follows. Set $(\gft[0],o_0)=(G,o)$. Give each vertex $v$ a splitting time $\tau_v$, where splitting times are \iid $\clock$ variables. When $t=\tau_v$, replace $v$ with two new vertices $v_1,v_2$, and give each a splitting time of $t+\clock$. Add $\Po{\la/2}$ edges between $v_1$ and $v_2$. Moreover, replace each edge of the form $uv$ with one of the edges $uv_1,uv_2$ chosen uniformly at random. If $v$ was the root, update the root to be $v_1$ or $v_2$, each with probability $1/2$. All these random choices are made independently from each other. Set $(\gft,o_t)$ to be the resultant graph.\end{definition}
We will frequently consider a single-vertex starting graph; we write $\gftc$ in this case.
\begin{remark}\label{geo-yule}The number of vertices of $\gftc$ over time, which is independent of all edge-related events, is a Yule process with rate $r=1$, that is, a pure birth process where the birth rate is $r$ times the population. Its value at time $t$ has a geometric distribution with mean $\ee{rt}$; see \cite[Section XVII.3]{Feller}.\end{remark}

\begin{definition}\label{cluster}The \emph{cluster process} $(G_t,o_t)$ with parameter $\la$ is the rooted connected multigraph formed by the component of the root in $\gft$.\end{definition}
It is natural to think of the cluster process as a reproduction process where individuals die when they leave the component of the root. In this sense it resembles a general branching process, or Crump--Mode--Jagers process, (see e.g.\ \cite[Chapter 6]{jagers}); however, these processes assume independence of the lifespans of different individuals, whereas in our model death events are highly interdependent.

We prove three main results about these processes, listed below.
\begin{theorem}\label{limit}
For each $\la>0$ there is a unique random rooted connected multigraph with finite expected root degree, $(M(\la),o)$, which is invariant under the cluster process in the sense that $(M(\la)_t,o_t)$ has the same distribution for any $t\geq 0$.
\end{theorem}
It is not immediately obvious that $M(\la)$ is almost surely finite. However, we prove a much stronger result.
\begin{theorem}\label{finite}
$\mean{\abs{M(\la)}}<\infty$ for every $\la>0$.
\end{theorem}
When considering the full process, a natural question is when it becomes disconnected, or equivalently when the full and cluster processes first differ.
\begin{theorem}\label{threshold}The time $t=\la$ is a sharp threshold for both connectedness of $\gftc$ and the existence of isolated vertices, that is, for any $\eps>0$, with high probability as $\la\to\infty$ the graph $\gftc[(1-\eps)\la]$ is connected but $\gftc[(1+\eps)\la]$ is disconnected with isolated vertices.\end{theorem}

\subsection{Questions}
In \cite{PIGG} we conjectured that $\mean{\abs{G(\la)}} \sim \la^{c\la}$ in agreement with computer simulation data. Simulations on $\mean{\abs{M(\la)}}$ showed a similar behaviour to $\mean{\abs{G(\la)}}$, and the same conjecture can be made. We know that $\mean{\abs{G(\la)}}$ is an analytic function of $\la$ because of results in percolation theory \cite{analyticity}. For $\mean{\abs{M(\la)}}$ we do not even have a proof of continuity. 
Apart from obtaining more detailed results about the behaviour of $M(\la)$, it would also be interesting to modify our splitting rule in order to obtain other random graph models with temporal invariance.

\section{Convergence to a limit}\label{secAsync}
In this section we prove \Tr{limit}; throughout the section we assume the parameter $\la>0$ is fixed.
Let $(G,o)$ be a random rooted graph such that $\mean{d(o)}$ is finite. Let $(G^\circ,o)$ be the single-vertex loopless graph with the same root $o$. Run the cluster process $(G_t,o_t)$ given in \Dr{cluster}, and let $H_t$ be the subgraph of $G_t$ induced by descendants of $o$. Note that $o_t\in H_t$ and $(H_t,o_t)$ evolves according to the law of the cluster process $(G^\circ_t,o_t)$, so has the same distribution.

\begin{lemma}\label{onevertex}With probability $1$, for sufficiently large $t$ we have $(G_t,o_t)=(H_t,o_t)$.\end{lemma}
\begin{proof}
We refer to edges of $\gft$ which were added after time $0$ as \textit{new} edges, and those which correspond (after replacements when vertices split) to edges of $G$ as \textit{old} edges. Let $e\in E(G)$ be an edge from the root, and let the corresponding edge at time $t$ meet $o'_t$, where $o'_t$ is a descendant of the root. We say that $e$ has been \textit{killed} by time $t$ if, for some $s\leq t$, we have $o'_s\neq o_s$ and no new edges meet $o'_s$. If $e$ has been killed by time $t$, then at time $s$ all paths from $o_s$ to $o'_s$ must use at least one old edge, and this property is preserved by splitting events, so the same is true for $t$. Thus, if a path from the root in $\gft$ uses any old edge, the first old edge in that path must not have been killed by time $t$, meaning that the old edges which have not been killed by time $t$ form a cut separating $H_t$ from the rest of $G_t$. It therefore suffices to show that with probability $1$ eventually every old edge has been killed.
	
For a specified edge $e$, consider the first time that the root splits and $o'_t\neq o_t$; call this $t_1$. Let $t_2,t_3,\ldots$ be the subsequent times that $o'_t$ splits, and let $X_k$ be the number of new edges meeting $o'_{t_k}$. Then $X_{k+1}\sim\operatorname{Bin}(X_k,1/2)+\Po{\la/2}$. This gives an irreducible Markov chain on $\mathbb N$ with a stationary distribution $\Po{\la}$. As a result, the chain is positive recurrent and in particular hits $0$ in finite time, killing $e$, with probability $1$. Since there were finitely many old edges, all of them are killed in finite time with probability $1$.\end{proof}

Before proceeding to the proof of \Tr{limit}, we first recall the \textit{Poisson edge model} of \cite{PIGG}. This is a long-range percolation model on the leaves of the canopy tree. We may label the complete binary trees of height $0, 1, \ldots$ in such a way that each tree is a subtree of the next, with each leaf also being a leaf of the next tree. The (binary) \textit{canopy tree} is then the union of this sequence of trees, and has an infinite sequence of leaves. The Poisson edge model is a random multigraph whose vertices are the leaves of the canopy tree, and whose edges are given by independently placing $\Po{2^{1-d(x,y)}\la}$ edges between each pair of leaves $x,y$, where $d(x,y)$ is the graph distance on the canopy tree. In \cite{PIGG} it is shown that the unique random rooted connected multigraph having finite expected root degree which is invariant under the synchronous version of the cluster process is given by the cluster of the root in the Poisson edge model. For the cluster process of \Dr{cluster}, the picture will be more complicated. Note that we may define the $T$-Poisson edge model for any binary tree $T$ in the same way: it is the random multigraph on the leaves of $T$, with $\Po{2^{1-d_T(x,y)}\la}$ edges independently between each pair of leaves $x,y$. We shall need a simple observation about the $T$-Poisson edge model.

Let $T$ be any binary tree, and fix an edge $uv$. We say that an edge of the $T$-Poisson edge model \textit{crosses} $uv$ if its endpoints are in different components of $T-uv$.
\begin{lemma}\label{edge-cross} The probability that the $T$-Poisson edge model has no edges which cross $uv$ is at least $\ee{-\la}$.\end{lemma}
\begin{proof}Write $L_u$, $L_v$ for the leaves of the components containing $u$ and $v$ respectively. The number of such edges is $\Po{z\la}$ where 
\begin{align*}z&=\sum_{x\in L_u}\sum_{y\in L_v}2^{1-d_T(x,y)}\\
&=\biggl(\sum_{x\in L_u}2^{-d_T(x,u)}\biggr)\biggl(\sum_{y\in L_v}2^{-d_T(v,y)}\biggr).
\end{align*}
We must therefore check that $z\leq 1$. Consider a random walk on the component of $T-uv$ containing $u$ started at $u$ and constrained to increase the distance from $u$ at every step, stopping if it reaches a leaf. Then for $x\in L_u$ the probability this walk stops at $x$ is $2^{-d_T(x,u)}$, since there are two possible moves at each step. Thus $\sum_{x\in L_u}2^{-d_T(x,u)}\leq 1$, and the same argument applies to $L_v$, giving the result.
\end{proof}
\begin{remark}In fact provided that $T$ has countably many ends we have equality in \Lr{edge-cross}, since both walks terminate almost surely.\end{remark}
\begin{proof}[Proof of \Tr{limit}]
We will construct a random multigraph $(M(\la),o)$ with the property that $(M(\la)_t,o_t)$ has the same distribution for any $t\geq 0$. To show uniqueness, we will show that $(G^{\circ}_t,o_t)$ converges in distribution to $(M(\la),o)$, and apply \Lr{onevertex}.
	
Our construction of $(M(\la),o)$ will use the $T$-Poisson edge model, working with a random tree $T$. (This tree can be thought of as the local limit of the Yule tree at time $t$, or equivalently the ball of radius $t$ in first passage percolation on the full binary tree with $\clock$ edge costs, after re-rooting at the leaf reached by a simple forward random walk from the root.)
	
To begin with, we construct some finite random trees $T(t)$ that will form the building blocks in the construction of $T$. Given a parameter $t>0$, we define a random rooted binary tree $T(t)$ as follows. Start from a single-vertex rooted tree, with an exponential clock of rate $1$ on the root. Whenever a clock on a vertex $v$ rings, add two children of $v$, each with their own independent exponential clocks of rate $1$ (do not replace the clock on $v$; each vertex rings at most once). Continue until time $t$. Note that $T(t)$ is almost surely finite.
Next we construct an infinite random tree $T$. Start from an infinite path $P=v_0v_1\cdots$, and label its edges with an infinite sequence $s_1,s_2,\ldots$ of \iid $\clock$ random variables. For each $i>0$, sample a copy $T_i$ of $T(\sum_{j\leq i}s_j)$, denote its root by $w_i$, and join $T_i$ to $P$ with the edge $v_iw_i$. Here each $T_i$ is sampled independently.

Having constructed $T$, consider the $T$-Poisson edge model. We let $M(\la)$ be the component of $v_0$ in this random multigraph, and let $v_0$ be the root of $M(\la)$. For $n\in\mathbb N$, let $L_n$ be the leaves of the component of $T-v_nv_{n+1}$ containing $v_n$.
\begin{claim}\label{as-finite}With probability $1$, $V(M(\la))\subseteq L_n$ for $n$ sufficiently large.\end{claim}
\begin{poc}
Starting from $k=0$, iteratively reveal the number of edges of the $T$-Poisson edge model between pairs of vertices until an edge crossing $v_kv_{k+1}$ is found. If this happens, update $k$ to be the smallest value such that no edge yet revealed crosses $v_kv_{k+1}$ and continue revealing. By \Lr{edge-cross}, for each different value of $k$ considered there is a probability of at least $\ee{-\la}$ that no suitable edge is ever found, no matter what was previously revealed. Thus almost surely one of the edges $v_kv_{k+1}$ is not crossed, meaning that $V(M(\la))\subseteq L_k$.
\end{poc}
Thus $M(\la)$ almost surely contains vertices from finitely many of the subtrees $T_i$. In particular, since each $T_i$ is almost surely finite, so is $M(\la)$.
\begin{claim}\label{claim-for-fig}$(M(\la)_t,o_t)$ has the same distribution as $(M(\la),o)=(M(\la)_0,o_0)$.\end{claim}
\begin{poc}Recall that the construction of $M(\la)$ was based on the randomly edge-labelled path $P$. Let us denote by $G(P,\la)$ the random graph constructed from any path $P$ with edges bearing positive real labels by following the above procedure. To compare $M(\la)$ with $M(\la)_t$, we will express the latter as $G(P_t,\la)$ for an appropriate randomly labelled path $P_t$: consider a Poisson point process $R=(-t_1, -t_2, \ldots, -t_k), k\geq 0$ on the interval $[-t,0]$ (where we assume that $t_i\geq t_{i+1}$) governed by Lebesgue measure and with duration $1$. We obtain $P_t$ from $P$ as follows. We change the label $s_1$ of the first edge of $P$ into $s_1+t_k$ if $k\geq 1$, or into $s_1+t$ if $k=0$. Moreover, we append $k$ edges at the start of $P$, and label them as follows. The first edge is labelled $t-t_1$, and for $i=2,\ldots, k$, the $i$th edge is labelled $t_{i-1}-t_i$. It is straightforward to check that $G(P_t,\la)$ is identically distributed with $(M(\la)_t,o_t)$ by identifying the times at which the root is split with the reversal $t_k, \ldots, t_2, t_1$ of $R$, using the fact that $t_{i-1}-t_i$  has distribution $\clock$, and so do $t_k$ and $t-t_1$.

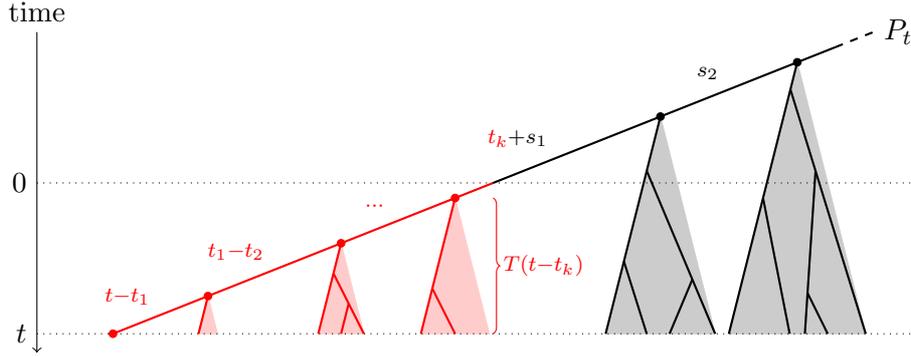
\begin{figure}[ht]
\centering
\begin{tikzpicture}
\node[red,anchor=south east] at (.625,.25) {$\scriptstyle t-t_1$};
\draw[fill,red!20!white] (1.25,.5) -- (1.125,0) -- (1.375,0) -- cycle;
\draw[fill,red] (1.25,.5) circle (0.05);
\draw[thick,red] (1.25,.5) -- (1.125,0);

\node[red,anchor=south east] at (2.125,.85) {$\scriptstyle t_1-t_2$};
\draw[fill,red!20!white] (3,1.2) -- (2.7,0) -- (3.3,0) -- cycle;
\draw[fill,red] (3,1.2) circle (0.05);
\draw[thick,red] (3,1.2) -- (2.7,0);
\draw[thick,red] (2.9,.8) -- (3.3,0);
\draw[thick,red] (3.1,.4) -- (3,0);

\node[red,anchor=south east] at (3.75,1.5) {$\scriptstyle \cdots$};
\draw[fill,red!20!white] (4.5,1.8) -- (4.95,0) -- (4.05,0) -- cycle;
\draw[fill,red] (4.5,1.8) circle (0.05);
\draw[red,decorate,decoration={brace}] (5,1.8) -- (5,0) node[midway,right]{$\scriptstyle T(t-t_k)$};
\draw[thick,red] (4.5,1.8) -- (4.05,0);
\draw[thick,red] (4.2,.6) -- (4.5,0);

\node[anchor=south east] at (5.85,2.34) {$\scriptstyle \color{red}{t_k}\color{black}{+s_1}$};
\draw[fill,black!20!white] (7.2,2.88) -- (6.48,0) -- (7.92,0) -- cycle;
\draw[fill] (7.2,2.88) circle (0.05);
\draw[thick] (7.2,2.88) -- (6.48,0);
\draw[thick] (7.02,2.16) -- (7.92,0);
\draw[thick] (7.62,.72) -- (7.32,0);
\draw[thick] (6.72,.96) -- (7.02,0);

\node[anchor=south east] at (8.1,3.24) {$\scriptstyle s_2$};
\draw[fill,black!20!white] (9,3.6) -- (9.9,0) -- (8.1,0) -- cycle;
\draw[fill] (9,3.6) circle (0.05);
\draw[thick] (9,3.6) -- (8.1,0);
\draw[thick] (8.91,3.24) -- (9.9,0);
\draw[thick] (8.55,1.8) -- (8.9,0);
\draw[thick] (9.24,2.16) -- (9.1,0);
\draw[thick] (9.135,.54) -- (9.4,0);

\draw[dotted] (-1,0) node[anchor=east]{$t$} -- (10.5,0);
\draw[->]  (-1,4) node[anchor=south]{time} -- (-1,-.25);
\draw[thick,red] (0,0) -- (5,2);
\draw[thick] (5,2) -- (9.5,3.8);
\draw[thick,dashed] (9.5,3.8) -- (10,4) node[anchor=west]{$P_t$};
\draw[dotted] (-1,2) node[anchor=east]{$0$} -- (10.5,2);
\draw[fill,red] (0,0) circle (0.05);
\end{tikzpicture}
\caption{Construction of $M(\la)_t$ via $P_t$ (proof of Claim \ref{claim-for-fig}).}
\end{figure}
To finish the proof that $(M(\la)_t,o_t) = G(P_t,\la)$ has the same distribution as $(M(\la),o) = G(P,\la)$, it suffices to prove that $P_t$ has the same distribution as $P$. To prove this, note that we can sample the labels $s_1, s_2, \ldots$ of $P$ as a Poisson point process on the real axis $[0,\infty)$ governed by Lebesgue measure and with duration $1$. Similarly, we can sample the labels of $P_t$ as the gaps of a Poisson point process on $[-t,\infty)$. But these two Poisson point processes are identically distributed once we shift by $t$, as required.\end{poc}
Next, we show that  $G_t^\circ$ converges in distribution to $M(\la)$.
To begin with, we can obtain $G_t^\circ$ by a construction similar to that of $M(\la)$, by keeping track of the genealogical tree $T_t$ of the vertices of $G_t^\circ$: the vertex set of $T$ comprises all vertices that appeared throughout the process $G_s^\circ, 0\leq s \leq t$, and if a vertex $v$ was replaced with $v_1,v_2$ at some time $s\leq t$, we join $v$ with an edge to each of $v_1,v_2$. Note that the vertex set of $G_t^\circ$ is contained in the set of leaves of $T_t$. To sample the edges of $G_t^\circ$, we put $\Po{2^{1-d_{T_t}(x,y)}\la}$ parallel edges independently between any two leaves $x,y$ of $T_t$, and identify $G_t^\circ$ with the component of $o$ in the resulting multigraph.

The times $t_1,\ldots, t_k$ when the root of $G_t^\circ$ splits are, by definition, given by a Poisson point process on $[0,t]$ governed by Lebesgue measure on that interval. Consequently, the ``reversed'' sequence of times $t-t_k, \ldots, t-t_1$ has the same distribution as $t_1,\ldots, t_k$. Using this fact, we may equivalently construct $G_t^\circ$ using $t-t_k, \ldots, t-t_1$ as the splitting times of the root, while leaving the rest of the construction unchanged. This realisation of $G_t^\circ$ coincides, by definition, with the following construction. Start with a random path $P_t$ with $k$ edges $e_1,\ldots, e_k$, where as above $k$ is the number of splittings of $o$ in the time interval $[0,t]$,  labelling $e_i$ with the time gap $s_i= t_{k+1-i}- t_{k-i}$ if $i=2,\ldots,k$ or $s_i=t-t_k$ if $i=1$. Attach to the endvertex $v_i$ of $e_i$ an independent copy of $T(\sum_{j\leq i}s_j)$ as above, and finally define a random graph on the leaves of the resulting tree by taking the component of the root in its Poisson edge model.
	
Appropriately coupled, $M(\la)$ and $G_t^\circ$ therefore give the same result so long as $M(\la)$ does not reach the end of the finite path $P_t$ in the above construction. Write $E_n$ for the event that $M(\la)$ does not extend past $v_n$. Given $\eps>0$, choose $n$ such that $\prob{E_n}<\eps/2$ (which is possible by Claim \ref{as-finite}) and $t$ such that $\prob{\Po{t}<n}<\eps/2$. 

For any set of isomorphism classes of rooted connected graphs $\mathcal S$, we have 
\begin{align*}\prob{G_t^\circ\in\mathcal S}&\leq\prob{M(\la)\in\mathcal S\wedge E_n\wedge (s_1+\cdots+s_n<t)}+\prob{\cmp{E_n}}+\prob{s_1+\cdots+s_n\geq t}\\
&<\prob{M(\la)\in\mathcal S}+\eps,
\end{align*}
and
\begin{align*}\prob{G_t^\circ\in\mathcal S}&\geq\prob{M(\la)\in\mathcal S\wedge E_n\wedge (s_1+\cdots+s_n<t)}\\
&\geq\prob{M(\la)\in\mathcal S}-\prob{\cmp{E_n}}-\prob{s_1+\cdots+s_n\geq t}\\
&>\prob{M(\la)\in\mathcal S}-\eps.
\end{align*}
Thus $G_t^\circ$ converges in distribution to $M(\la)$ as $t\to \infty$. The uniqueness of $M(\la)$ now follows from \Lr{onevertex}, since if $G$ is a random graph with $G_t$ identically distributed for every $t$, that lemma implies that the distribution of $G$ is the limit of the distribution of  $G_t^\circ$.
\end{proof}

The random multigraph $M(\la)$ described above differs from the corresponding multigraph $G(\la)$ for the synchronous case studied in \cite{PIGG}, that is, the component of the root in the original Poisson edge model on the canopy tree. To see this, it is sufficient to consider the probability, conditional on $d(o)=2$, of a double edge from the root.
For $M(\la)$ this is $\sum_{x\neq o}4^{1-d(o,x)}$, where the sum is taken over all other leaves of the random tree $T$. 
Note that the probability that $w_1$ is a leaf is $\prob{\tau(w_1)<s_1}$, where $\tau(w_1)$ is the length of $w_1$'s clock. 
Since $\tau(w_1)$ and $s_1$ are i.i.d., we have $\prob{w_1\text{ a leaf}}=1/2$; clearly $w_i$ is less likely to be a leaf than $w_1$ if $i>1$, so each $w_i$ is a leaf with probability at most $1/2$. 
For each $i\geq 1$, the probability of a double edge to a descendent of $w_i$ is $4^{-i}$ if $w_i$ is a leaf, and at most $4^{-i-1}$ otherwise 
(being maximised when both its offspring are leaves). So the probability of a double edge is at most 
$\sum_{i\geq 1}(4^{-i}+4^{-i-1})/2=1/4$.
For the canopy tree version $G(\la)$, the probability of a double edge is $\sum_{h\geq 1}2^{h-1}4^{1-2h}=2/7$, and so $M(\la)$ has a strictly smaller double-edge probability.

\section{Finite expected size}
In this section, we consider the expected size $\mean{\abs{M(\la)}}$. While the expected size of $G(\la)$ is finite for every $\la>0$ \cite{PIGG}, it is not immediately clear whether the same is true of $M(\la)$. Since $M(\la)$ arises from the $T$-Poisson edge model on a random tree $T$, and we know that the expected cluster size is finite for the Poisson edge model on the canopy tree, and that the cluster size of the Poisson edge model on any binary tree is almost surely finite (Claim \ref{as-finite}), one might hope to prove a universal bound (depending on $\la$) on the expected cluster size for any binary tree, whence the desired result would follow by averaging. However, no such bound exists; indeed, there are binary trees on which the expected cluster size of the Poisson edge model is infinite for sufficiently large $\la$. One example may be obtained by replacing each edge of the canopy tree by a two-edge path with a pendant leaf attached to the new vertex. If $v$ was a leaf of the canopy tree at distance $2k$ from $o$, then the new tree contains a sequence of $2k+2$ leaves, starting at $o$ and ending at $v$, such that each consecutive pair is at distance $4$. Each of these pairs is adjacent in the Poisson edge model on this tree with probability $1-\ee{-\la/8}$, and so every such $v$ is in the component of $o$ with probability at least $(1-\ee{-\la/8})^{2k+1}$. Provided $\la\geq 8\log (2+\sqrt 2)$, it follows that the expected size of this component is infinite.

Of course, the initial sections of such a tree are not typical Yule trees, and so this example does not rule out the possibility of exploiting the large-scale structure of the tree $T$ constructed in the previous section. However, we will find it easier to use a more local approach: rather than showing that an initial section of $T$ is typically well-behaved everywhere, we work directly with Definition \ref{cluster} and explore the component of the root in $\gft$. This means that we only need $T$, which corresponds to the splitting events, to behave well in such parts as we encounter during this exploration. In the remainder of the section, we prove \Tr{finite}.

\subsection{Outline of proof}
Fix $\la>0$. Note that since $\cmot$ converges in distribution to $M(\la)$ and both $\cmot$ and $M(\la)$ are almost surely finite we have $\mean{\abs{\cmot}}\to\mean{\abs{M(\la)}}$ as $t\to\infty$. Our basic strategy is to prove a bound $f(t)$ on $\mean{\abs{\cmot}}$ which changes only slowly with $t$, and has a finite limit. Recall that $\cmot$ is the component of the root in $\gftc$. First we bound the size of $\cmot$ at time $t+\eps$.

\begin{lemma}\label{deriv}Fix times $t\geq 0$ and $\eps>0$, and let $X_\eps=\abs{\gftc[\eps]}$ be the total number of vertices in the full process at time $\eps$. Then we have
\[\mean{\abs{\cmot[t+\eps]}}<(1-\eps)\mean{\abs{\cmot}}+\eps\mean{\abs{\cmot[t+\eps]}\mid X_\eps=2}+4\eps^2\ee{t+\eps}.\]
\end{lemma}
\begin{proof}
Conditioning on the value of $X_\eps$, we have
\begin{align*}\mean{\abs{\cmot[t+\eps]}}&=\prob{X_\eps=1}\mean{\abs{\cmot[t+\eps]}\mid X_\eps=1}+\prob{X_\eps=2}\mean{\abs{\cmot[t+\eps]}\mid X_\eps=2}\\
&\phantom{=}+\prob{X_\eps>2}\mean{\abs{\cmot[t+\eps]}\mid X_\eps>2}.\end{align*}
Note that, conditioned on $X_\eps=1$, $\cmot[t+\eps]$ is just the result of letting the single vertex at time $\eps$ evolve for an additional time $t$, and $\prob{X_\eps=1}=\ee{-\eps}<1-\eps+\eps^2$, so
\begin{align*}\prob{X_\eps=1}\mean{\abs{\cmot[t+\eps]}\mid X_\eps=1}&<(1-\eps+\eps^2)\mean{\abs{\cmot}}\\
&<(1-\eps)\mean{\abs{\cmot}}+\eps^2\ee{t+\eps}.\end{align*}
Also, $\prob{X_\eps=2}<\eps$, which gives the required second term. 

To deal with the third term, recall from Remark \ref{geo-yule} that $X_\eps\sim\operatorname{Geo}(\ee{-\eps})$ and so $\prob{X_{\eps}>2}=(1-\ee{-\eps})^2<\eps^2$. Now suppose that $X_\eps>2$. This means that there is some random time $\eta_2<\eps$ at which the second splitting event occurs. Nothing that happens after $\eta_2$ can affect the event $X_\eps>2$, and so we may condition on $\eta_2$. At time $\eta_2$ there are three vertices, which may or may not be connected by edges. Certainly $\abs{\cmot[t+\eps]}$ is dominated by $\abs{\gftc[t+\eps]}$, which, conditioned on $\eta_2$, has expectation $3\ee{t+\eps-\eta_2}<3\ee{t+\eps}$. Thus the final term is less than $3\eps^2\ee{t+\eps}$, as required.
\end{proof}
Conditioned on $X_\eps=2$, $\gftc[t+\eps]$ is distributed as two independent copies of the full process run for time $t$ with some edges between them, rooted at the root of the first copy. We will show that the probability of some of these edges touching the component of the root in the first copy is exponentially small. If this does happen, we argue that the expected number of edges between the two copies is not much larger than its unconditional expectation (\ie $\la/2$), and that consequently we connect together (on average) not too many components. The main issue with this is that conditioning on this unlikely event might change the expected size of a component significantly, so we must control this. If we can do this, we will have shown that
\begin{equation}\label{expbound}\mean{\abs{\cmot[t+\eps]}\mid X_\eps=2}\leq(1+h(t))\mean{\abs{\cmot}},\end{equation}
where $h(t)$ is some function that decays exponentially in $t$.
It will follow, from \eqref{expbound} and Lemma \ref{deriv}, that for any fixed $t\geq 0$ we have,
\[\limsup_{\eps\to 0+}\frac{\mean{\abs{\cmot[t+\eps]}}-\mean{\abs{\cmot}}}{\eps}\leq h(t)\mean{\abs{\cmot}},\]
and so if $f:[0,\infty)\to[0,\infty)$ is a function satisfying $f(0)=1$ and $f'(t)=h(t)f(t)$, then $f(t)\geq\mean{\abs{\cmot}}$ for each $t$.
Now for this $f$ we have $\frac{\mathrm{d}}{\mathrm{d}t}\log f(t)=h(t)$ and so 
\[\lim_{t\to\infty}f(t)=\exp\int_0^\infty h(s)\mathrm{d}s<\infty.\]

Write $\dmot$ for the result of running the cluster process for time $t$ starting from two vertices with $N\sim\Po{\la/2}$ number of edges between. We consider the descendants of the two original vertices in the corresponding full process $\dmotf$ as two independent copies of $\gftc$, with the ``left'' copy being descendants of the original root, and say that the $N$ edges between the two copies are \emph{old}, and others are \emph{new}. Note that the component of the root in the subgraph induced by the left copy is distributed as $\cmot$. We follow what happens to the left-endpoints of all old edges, and to the root. 
Recall that an old edge is \emph{killed} by a splitting event if after that event its left-endpoint is not the root, and meets no new edges. Consider the following four events, for a fixed time $t$ and $0<\alpha<1$; note that some of these events may depend on what happens after time $t$.
\begin{enumerate}[A:]
\item the left-endpoint of some old edge splits less than $\alpha t$ times by time $t$.
\item after the left-endpoint of some old edge splits $\alpha t/3$ times, it is either the root or the left-endpoint of more than one old edge.
\item B does not occur, but some new edge meets the left-endpoint of some old edge for the entire period between the $(\alpha t/3)$th and $(2\alpha t/3)$th splits of the latter.
\item B and C do not occur, but some old edge is not killed between its $(2\alpha t/3)$th and $\alpha t$th splits.
\end{enumerate}

Writing $\mathbb I_A$, etc., for the indicator functions of these events, we have
\begin{align}\mean{\abs\dmot}&\leq\mean{\abs\dmot(\mathbb I_A+\mathbb I_B+\mathbb I_C+\mathbb I_D+\mathbb I_{\cmp{(A\cup B\cup C\cup D)}})}\nonumber\\
&\leq\sum_{E\in\{A,B,C,D\}}\mean{\abs\dmot\mid E}\prob{E}+\mean{\abs\dmot\mid \cmp{(A\cup B\cup C\cup D)}}.\label{combineABCD}\end{align}

\subsection{Dealing with event A}
Note that the left-endpoint of a given edge splits $\Po{t}$ times in time $t$, and
\begin{equation*}
\prob{\Po{t}\leq\floor{\alpha t}}\leq(\floor{\alpha t}+1)\frac{\ee{-t}t^{\floor{\alpha t}}}{\floor{\alpha t}!}=O(\sqrt t)(\ee{\alpha-1}\alpha^{-\alpha})^t.
\end{equation*}
Since $\lim_{\alpha\to0+}\alpha^\alpha=1$, we may choose $\alpha>0$ such that
\begin{equation}\prob{A}\leq\frac{\la}{2}\prob{\Po{t}\leq\floor{\alpha t}}=O\bigl(\ee{(\ee{-\la}/2-1)t}\bigr).\label{A-prob}\end{equation}

We next define a variant of the full process: the \textit{singleton-free process} $S_t$ starts from a single vertex with $\Po{\la/2}$ tokens. It proceeds as the full process with tokens distributed randomly between the offspring when a vertex splits, but with the exception that any vertex which is isolated and has no tokens is immediately discarded. 

First we will show that $\mean{\abs{S_t}}$ is bounded by the expected size of a Yule process of rate $r=r(\la)<1$. The intuition here is that each splitting event has at least a constant probability of producing an isolated vertex, and we can just ignore these events, resulting in a thinning of the rate by a constant factor. However, we need to be slightly careful to check that the lower bound on the probability of creating an isolated vertex still holds even conditioned on the splitting vertex not having been isolated at any point in its history. We will need the following lemma, which will be used again for the other events.
\begin{lemma}\label{poisson+}If $X\sim\Po m$ and $Y\sim\operatorname{Bin}(X,p)$ for some $p\in(0,1]$, then $X\mid (X\geq k)$ is stochastically dominated by $k+\Po m$.\end{lemma}
\begin{proof}We handle the case $p=1$; the case $p<1$ follows by noting that $Y\sim\Po{pm}$ and $X-Y$ are independent. We may sample $X\mid (X\geq k)$ by repeatedly sampling $X$, keeping the first value which is at least $k$. Since we can take the $r$th sample of $X$ as the number of points occurring in the interval $[r-1,r]$ in a Poisson process of intensity $m$, this is the same as letting the Poisson process run until the first time we have seen $k$ points since the last integer, then continuing until the next integer. 
This is clearly dominated by letting the process run to the first time we have seen $k$ points since the last integer, then continuing for time $1$, which gives the required distribution.\end{proof}
If we only have the weaker condition that $X\lest\Po{m}$ we cannot get any bound on $\mean{X\mid Y\geq 1}$, but the following bounds are sufficient for our purpose.
\begin{corollary}\label{sub-poisson}Suppose $X\lest\Po m$ is nonnegative, and $Y\sim\operatorname{Bin}(X,p)$ for some $p\in(0,1]$. Then $\prob{Y\geq 1}\leq pm$ and $\prob{Y\geq 1}\mean{X\mid (Y\geq 1)}\leq pm(m+1)$.
\end{corollary}
\begin{proof}We can couple $X$ with a variable $X'\sim\Po m$, and couple $Y$ with $Y'\sim\operatorname{Bin}(X',p)$ in the natural way, so that $(Y\geq 1)\subseteq(Y'\geq 1)$. Then, since $0\leq X\leq X'$ we have
\begin{align*}\prob{Y\geq 1}\mean{X\mid Y\geq 1}&\leq\prob{Y\geq 1}\mean{X'\mid Y\geq 1}\\
&\leq \prob{Y'\geq 1}\mean{X'\mid Y'\geq 1}.\end{align*}
Lemma \ref{poisson+} gives $\mean{X'\mid Y'\geq 1}\leq m+1$, and $\prob{Y\geq 1}\leq\prob{Y'\geq 1}\leq\mean{Y'}=mp$.
\end{proof}
We may sample $S_t$ by using a Yule process $Y_t$ of rate $1$ for the splitting events, then determining the movement of new edges and removing any vertices which were isolated at any point in their history. Similarly, we can simulate a Yule process $Y^{(r)}_t$ of rate $r<1$ from the same copy of $Y_t$ by, independently for each splitting event, removing all descendants of one offspring with probability $1-r$. Conditional on $Y_t$, each vertex which has undergone $k$ splitting events has probability $\bigl(\frac{r+1}{2}\bigr)^k$ of surviving in $Y_t^{(r)}$. In $S_t$, conditional on a vertex having survived $j$ splits without being isolated, we argue by induction on $j$ that the number of edge-ends meeting it is dominated by $\Po{1+\la}$. This is true for $j=0$. Assuming the statement holds for $j$, a vertex which has split $j+1$ times may inherit the first edge-end from its parent, and receives at most $\Po{\la}$ other edge-ends; conditioning on not being isolated does not change the number of additional edge-ends if it did inherit the first edge, and increases it to at most $1+\Po{\la}$ if not. So the result holds for all $j$. Consequently, given that a vertex has survived $j$ splits without being isolated, its offspring after the next split have at most $\operatorname{Ber}(1/2)+\Po{\la}$ edge-ends, so are each isolated with probability at least $\frac{\ee{-\la}}{2}$; it follows that each vertex in $Y_t$ which underwent $k$ splitting events has probability $\bigl(\frac{2-\ee{-\la}}{2}\bigr)^k$ of surviving in $S_t$. For $r=1-\ee{-\la}$, each vertex has a higher probability of surviving in $Y^{(r)}_t$ than $S_t$, and so we have \[\mean{\abs{S_t}}\leq\mean{\abs{Y^{(r)}_t}}=\ee{(1-\ee{-\la})t}.\]

Lemma \ref{poisson+} implies that $N\mid A \lest 1+\Po{\la/2}$. To see this, note that we may first condition on the tree of splitting events. For each possible tree $T$, each old edge independently has some probability $p_T$ of following a path in the tree which splits fewer than $\alpha t$ times; we may ignore trees for which $p_T=0$. Thus $N\mid T,A\lest1+\Po{\la/2}$, and the result follows by averaging over $T$.

Now we consider the singleton-free process conditional on $A$. Suppose a vertex meeting an old edge or root splits, and one of the new vertices created, $v$, does not meet a new edge or the root. Conditioning on $A$ does not affect the future evolution of $v$, and it evolves as the non-root half of a singleton-free process (or is discarded if it has no new edges). Thus the expected number of descendants of $v$ is at most $\mean{\abs{S_t}}/2=\ee{(1-\ee{-\la})t}$. For each old edge, the expected number of times it splits is $t$ before conditioning on A, and cannot increase after conditioning; the same applies to the root. Thus the expected number of times such a vertex is created is at most $(2+\la)t$. Since every vertex at time $t$ is either a descendant of such a vertex, meets an old edge, or is the root, we have
\[\mean{\abs{\dmot}\mid A}\leq\mean{\abs{S_t}\mid A}\leq2+\la+(2+\la)t\ee{(1-\ee{-\la})t},\]
and so (recalling \eqref{A-prob}) we have
\begin{equation}\prob{A}\mean{\abs{\dmot}\mid A}=O(t\ee{-t\ee{-\la}/2}).\label{A-total}\end{equation}

\subsection{Dealing with event B}
Since there are $X\sim\Po{\la/2}$ left-endpoints of old edges and one root, and each pair has probability $2^{-\alpha t/3}$ of coinciding after $\alpha t/3$ splits, a union bound gives
\begin{equation}\prob{B}\leq\mathbb{E}\biggl(\binom{X+1}2\biggr)2^{-\alpha t/3}=\biggl(\frac{\la^2}{8}+\frac{\la}{2}\biggr)2^{-\alpha t/3}.\label{B-prob}\end{equation}

Consider the full tree of possible locations for left-endpoints after $\alpha t/3$ splits. Order these locations $v_1,\ldots,v_{2^{\alpha t/3}}$; without loss of generality we may assume the root is at $v_1$ after $\alpha t/3$ splits. Writing $X_i$ for the number of old edges at location $v_i$ after $\alpha t/3$ splits, the $X_i$ are \iid $\Po{2^{-\alpha t/3}\la/2}$ random variables.
We will control the expected number of old edges conditioned on $B$. $B$ occurs if and only if either $X_1\geq 1$ or $X_i\geq 2$ for some $i>1$. For each $i\geq 2$, let $B_i$ be the event that $X_i\geq 2$, and $X_j\leq 1$ for each $j>i$. Let $B_1$ be the event that $X_1\geq 1$ but $X_j\leq 1$ for each $j>1$. Now the events $(B_i)_{i=1}^{2^{\alpha t/3}}$ form a partition of $B$, and $N=\sum_{j=1}^{2^{\alpha t/3}} X_j$. \Lr{poisson+} gives $\mean{X_i\mid B_i}\leq\mean{X_i}+2$, and $\mean{X_j\mid B_i}\leq\mean{X_j}$ if $j\neq i$,
so $\mean{N\mid B_i}\leq \mean{N}+2$ for each $i$. Thus 
\[\mean{N\mid B}=\sum_{i=1}^{2^{\alpha t/3}}\prob{B_i\mid B}\mean{N\mid B_i}\leq\la/2+2.\]
The old edges therefore combine, on average and conditional on $B$, at most $\la/2+3$ components from the two copies. Since $B$ does not depend on splitting times or new edges, each component has expected size $\mean{\abs{\cmot}}$. Thus, recalling \eqref{B-prob}, we have
\begin{equation}\prob{B}\mean{\abs{\dmot}\mid B}=O(2^{-\alpha t/3})\mean{\abs{\cmot}}.\label{B-total}\end{equation}
\subsection{Dealing with event C}\label{sec:eventC}
Randomly designate one end of each new edge to be the ``head'', and the other the ``tail'', so that the number of edges $xy$ with head $x$ and the number with head $y$ are independent. We set $C_{\mathrm h}$ ($C_{\mathrm t}$) to be the event that the head (the tail) of some new edge coincides with the left end of some old edge for the period in question. Since $C=C_{\mathrm h}\cup C_{\mathrm t}$ and by symmetry of $C_{\mathrm h},C_{\mathrm t}$, we have $\prob{C}\mean{\abs{\dmot}\mid C}\leq2\prob{C_{\mathrm h}}\mean{\abs{\dmot}\mid C_{\mathrm h}}$.

We first condition on $\cmp B$; since $\prob{B\mid N=n}$ is increasing in $n$, $(N\mid \cmp{B})\lest\Po{\la/2}$. Given $\cmp{B}$, each of the $N\mid\cmp{B}$ old edges coincides with the head of a new edge for the period between its $(\alpha t/3)$th and $(2\alpha t/3)$th splits independently and with equal probability $p$. 
Since the number of heads coinciding with a given old edge at the start of the period is distributed $\Po{\kappa}$ for some fixed $\kappa\leq\la/2$, a union bound gives $p\leq2^{-\alpha t/3}\la/2$. Thus
\begin{equation}\prob{C}\leq2\prob{C_{\mathrm h}\mid \cmp B}\leq2^{-\alpha t/3}\la^2/2.\label{C-prob}\end{equation}
Also, Corollary \ref{sub-poisson} gives $\prob{C_{\mathrm h}\mid \cmp{B}}\mean{N\mid C_{\mathrm h}}\leq 2^{-\alpha t/3}(\la/2)^2(1+\la/2)$.

Next we bound $\prob{C_{\mathrm h}}\mean{\abs{\dmot}\mid C_{\mathrm h}}$. Lemma \ref{poisson+} implies that the number of heads which coincide with a given old edge, after conditioning on $C_{\mathrm h}$, is dominated by $1+\Po{p\kappa}$, and the number of other new edges is independent of $C_{\mathrm h}$. Thus we may couple the new edges conditioned on $C_{\mathrm h}$ as a subgraph of the unconditioned new edges together with at most $N\mid C_{\mathrm h}$ additional new edges. The expected size of the component of a given vertex in the subgraph of unconditional new edges is $\mean{\abs{\cmot}}$, and since the old edges and additional new edges merge at most $2\mean{N\mid C_{\mathrm h}}+1$ components on average, 
\begin{align}\prob{C}\mean{\abs{\dmot}\mid C}&\leq2\prob{C_{\mathrm h}}(2\mean{N\mid C_{\mathrm h}}+1)\mean{\abs{\cmot}}\nonumber\\
&=\mean{\abs{\cmot}}O(2^{-\alpha t/3})\label{C-total}.\end{align}

\subsection{Dealing with event D}
Now suppose that $B$ and $C$ do not occur. Again, we have $N\mid \cmp{B}\cap\cmp{C}\lest\Po{\la/2}$. Since $C$ does not occur, all new edges that meet left ends of old edges after $2\alpha t/3$ splits were created after the $(\alpha t/3)$th split, and since $B$ does not occur, each of these meets only one old edge. Thus every old edge is not killed between its $(2\alpha t/3)$th and $\alpha t$th splits independently with some probability $p$. Corollary \ref{sub-poisson} therefore gives $\prob{D}\mean{N\mid D}\leq p(\la/2)(\la/2+1)$.

We next bound $p$. Note that being killed is monotone on adding new edges. Suppose that after a given split an old edge $e$ meets $X_0\lest 1+\Po{\la}$ new edges. Adding extra new edges, if necessary, we may assume $e$ meets $1+\Po{\la}$ new edges. After the next split, conditioned on $e$ meeting at least one new edge, we claim that it meets $X_1\lest 1+\Po{\la}$ new edges. If the first of the $1+\Po{\la}$ new edges still meets $e$, there are $\Po{\la/2}+\Po{\la/2}\sim\Po{\la/2}$ other new edges meeting $e$, whereas if not we have $\Po{\la}$ edges meeting $e$, conditioned to be positive, and by Lemma \ref{poisson+} this is dominated by $1+\Po{\la}$.
Thus conditioning on not having been killed at the previous step leaves at most $1+\Po\la$ new edges meeting $e$, giving a probability of at least $\frac12\ee{-\la/2}$ of being killed at the next step; write $c_\la=1-\frac12\ee{-\la/2}$. It follows that $p\leq c_\la^{\alpha t/3}$ and so
\begin{equation}\prob{D}\leq \mean{N\mid \cmp B\cap \cmp C}p\leq c_\la^{\alpha t/3}\la/2.\label{D-prob}\end{equation}

For each old edge $e$, we associate each new edge $e'$ which meets $e$ at any point between its $(2\alpha t/3)$th and $\alpha t$th splits with the interval for which it meets $e$, \ie the set of indices in $\{\alpha t/3,\ldots,\alpha t\}$ of splits after which $e$ and $e'$ meet. Denote the number of new edges meeting $e$ for an interval $I$ by $X_{e,I}$; note that $X_{e,I}\sim\Po{\kappa_I}$ for some $\kappa_I$ depending only on $I$, and all these are independent. We now condition on the number of old edges and which pairs $e,I$ have $X_{e,I}\geq 1$; this is sufficient information to determine whether $D$ occurs. Lemma \ref{poisson+} gives $X_{e,I}\mid D\lest1+\Po{\kappa_I}$ for each $e,I$.
We can thus couple the new edges conditioned on $D$ as a subgraph of the unconditioned new edges together with at most $(N\mid D)(2\alpha t/3)^2$ additional new edges. As in Section \ref{sec:eventC}, it follows that
\begin{equation}\prob{D}\mean{\abs{\dmot}\mid D}=\mean{\abs{\cmot}}O(t^2c_\la^{\alpha t/3}).\label{D-total}\end{equation}
\subsection{Final bounds}
If none of $A,B,C,D$ occur then all old edges have been killed. Since this means any path from the root uses only new edges (see the proof of \Lr{onevertex}), the component of the root is entirely within the left half, and thus we have 
\begin{align*}\mean{\abs\dmot\mid \cmp{(A\cup B\cup C \cup D)}}&=\mean{\abs{\cmot}\mid \cmp{(A\cup B\cup C \cup D)}}\\
&\leq\mean{\abs{\cmot}}/\prob{\cmp{(A\cup B\cup C \cup D)}}.\end{align*} 
Combining \eqref{A-prob}, \eqref{A-total}, \eqref{B-prob}, \eqref{B-total}, \eqref{C-prob}, \eqref{C-total}, \eqref{D-prob} and \eqref{D-total} using \eqref{combineABCD}, we have 
\[\mean{\abs\dmot}=(1+o(\zeta^t))\mean{\abs{\cmot}},\]
for some $\zeta<1$, as required for \eqref{expbound}, which thus completes the proof of finiteness.

\section{A sharp threshold for connectedness}\label{thresasync}

In this section we prove \Tr{threshold}, giving a sharp threshold for connectedness of $\gftc$. We show that, as for the binomial random graph, it coincides with the threshold for isolated vertices to appear. Our methods in this section will follow those of \cite{PIGG} closely. 
For both directions we will need the following simple concentration bound.
\begin{lemma}\label{leaves}
Let $f(t):(0,\infty)\to(0,\infty)$ be any function with $f(t)\to\infty$ as $t\to\infty$. Then with high probability we have $\ee{t-f(t)}<\abs{\gftc}<\ee{t+f(t)}$.\end{lemma}
\begin{proof}
Set $n_1=\ceil{\ee{t-f(t)}}$ and $n_2=\floor{\ee{t+f(t)}}$. Since $\abs{\gftc}\sim\operatorname{Geo}(\ee{-t})$, we have
\[\prob{n_1<\abs{\gftc}<n_2}=(1-\ee{-t})^{n_1}-(1-\ee{-t})^{n_2-1}\underset{t\to\infty}{\longrightarrow} 1.\qedhere\]
\end{proof}
We first show that isolated vertices appear soon after time $\la$.
\begin{proposition}\label{async-iso}
Let $f(t):(0,\infty)\to(0,\infty)$ be any function with $f(\la)\to\infty$ as $\la\to\infty$. If $t>\la+f(\la)$ then as $\la\to\infty$ with high probability $\gftc$ has an isolated vertex.
\end{proposition}
\begin{proof}
For technical reasons we prove the same statement for the modified process obtained by adding $\Po{\la/2}$ extra edges at the first splitting event. This ensures that each vertex has the same probability $\ee{-\la}$ of being isolated. Conditioned on the tree $T(t)$ defined in the proof of \Tr{limit}, \cite[Lemma 7.1]{PIGG} applies and gives $\prob{X\mid T(t)}\leq 2/(2+\abs{\gftc}\ee{-\la})$, where $X$ is the event that no vertex is isolated.
Note that $\la+f(\la)/2<t-g(t)$, where $g(t)$ is another function satisfying $g(t)\to\infty$. By \Lr{leaves}, with high probability $\abs{\gftc}>\ee{t-g(t)}$;
conditional on this we have $\prob{X}\leq2/(2+\ee{f(\la)/2})=o(1)$.
\end{proof}
To complete the proof of \Tr{threshold}, we must show that with high probability $\gftc(\la)$ is connected shortly before $t=\la$.
For this we need another result from \cite{PIGG}, but first we define some terms used. Fix a finite binary tree $T$ representing descendants of a marked apex vertex, and $k\in\mathbb N$. We say that two vertices are \textit{siblings} if they have the same parent, and two pairs of siblings are \textit{$k$-cousins} if they have a common ancestor which is at most distance $k$ on $T$ from all of them. Let $G$ be a graph whose vertices are leaves of $T$. We say two siblings $x,y$ are \textit{strongly linked} by $G$ if $G$ contains an edge between a descendant of $x$ and a descendant of $y$, and \textit{weakly linked} by $G$ if there is some vertex $z$ of $T$ which is a sibling of one of the $k$ lowest ancestors of $x,y$, such that $G$ contains edges between a descendant of $x$ and one of $z$, and between a descendant of $y$ and one of $z$. \cite[Lemma 7.2]{PIGG} says that the following set of conditions is sufficient for $G$ to be connected:
\begin{enumerate}[(i)]\item\label{siblings} every pair of siblings in $T$ is either strongly linked or weakly linked by $G$;
\item\label{cousins} of every two pairs which are $k$-cousins, at least one is strongly linked by $G$;
\item\label{patriarchs} any pair of siblings within the top $k$ layers of $T$ are strongly linked by $G$.\end{enumerate}
\begin{proposition}\label{async-conn}For any $\alpha>1$, if $t\leq\la-\alpha\log\la$ then as $\la\to\infty$ with high probability $\gftc$ is connected.\end{proposition}
\begin{proof}
Regarding $\gftc$ as a random graph on the leaves $L(t)$ of $T(t)$, we will show the conditions above hold with high probability, for some suitable $k$. Choose $\alpha'>0$ such that $\alpha-\alpha'>1$; then, by \Lr{leaves}, with high probability $\abs{L(t)}<\ee{t+\alpha'\log t}<\ee{t+\alpha'\log\la}$, so
\begin{equation}\label{treesize}\abs{T(t)}<2\ee{t+\alpha'\log \la}.\end{equation} 

Suppose \eqref{treesize} holds, and set $k=\log_2\la$. The probability that a particular pair of siblings fails to be strongly linked is $\ee{-\la/2}$, and since each pair of siblings has at most $k2^k=\la\log_2\la$ pairs of $k$-cousins, the total number of ways to choose two pairs of siblings which are $k$-cousins is at most $\ee{t+\alpha'\log \la}\la\log_2\la=\ee{t+(1+\alpha')\log \la}\log_2\la=o(\ee{\la})$. For each such choice, the probability that neither pair is strongly linked by $\gftc(\la)$ is $\ee{-\la}$ and so with high probability \eqref{cousins} holds. There are at most $\la$ pairs of siblings in the top $k$ layers of $T(t)$, and so \eqref{patriarchs} also holds with high probability.
Finally, for a fixed pair of siblings below this point the probability that they are not strongly or weakly linked by $\gftc$ is 
\[\ee{-\la/2}\bigl(1-(1-\ee{-\la/4})^2\bigr)\cdots\bigl(1-(1-\ee{-\la/2^{k-1}})^2\bigr)<2^{k-1}\ee{-\la(1-2^{-k})}.\]
Thus the probability that some pair fails to be strongly or weakly linked is at most
\begin{align*}2^k\ee{-\la(1-2^{-k})}\ee{t+\alpha'\log\la}&=\la\ee{(t+\alpha'\log\la)-(t+\alpha\log\la)(1-1/\la)}\\
&=O(\la^{1+\alpha'-\alpha})=o(1).\qedhere\end{align*}
\end{proof}

\section*{Acknowledgements}
Both authors were supported by the European Research Council (ERC) under the European Union's Horizon 2020 research and innovation programme (grant agreement no.\ 639046). J.H. was also partially supported by the UK Research and Innovation Future Leaders Fellowship MR/S016325/1. We are grateful to the anonymous referee for their very helpful comments.

\end{document}